\documentclass[12pt,a4paper,oneside]{article}
\usepackage{amsmath,amstext,amssymb,amsthm}
\usepackage{xcolor}

\newcommand{\expr}[1]{\left( #1 \right)}

\newcommand{\pr}{\mathbf{P}}

\newcommand{\eps}{\varepsilon}

\newtheorem{thm}{Theorem}[section]
\newtheorem{lem}[thm]{Lemma}
\newtheorem{prop}[thm]{Proposition}

\newtheorem{rem}[thm]{Remark}

\newcommand{\pref}[1]{(\ref{#1})}
\newcommand{\Rd}{\mathbb{R}^d}

\DeclareMathOperator{\im}{Im}

\title{Potential theory of one-dimensional geometric stable processes}
\author{ T. Grzywny, M. Ryznar \\
 Institute of Mathematics and Computer Sciences,\\
  Wroc\l{}aw University of Technology, Poland}

\date{}
\begin{document}
\maketitle

\begin{abstract}
The purpose of this paper is to find optimal estimates for the
Green function and the Poisson kernel for a half-line and  intervals of {\it the geometric
stable process} with parameter $\alpha\in(0,2]$. This process has an infinitesimal generator of the form
 $-\log(1+(-\Delta)^{\alpha/2})$. As an application we prove the scale invariant Harnack inequality as well as the boundary Harnack principle.
\end{abstract}



\section{Introduction}

 Let $B = (B_t , t \ge 0)$ be a Brownian motion in $\Rd$ and $ T = (T_t : t \ge 0)$ be
a subordinator independent of $ B$. The process $X = (X_t : t \ge 0)$ defined by $X_t = B_{T_t}$ is a
rotationally invariant L\'{e}vy process in $\Rd$ and is called a subordinate Brownian motion.
The subordinator $T$ used to define the subordinate Brownian motion $X$ can be interpreted as
{\em operational} time or {\em intrinsic} time. For this reason, subordinate Brownian motions have been
used in mathematical finance and other applied fields.

 Let $\psi$ denote
the Laplace exponent of the subordinator $T$, that is,
$$E \exp\{-\lambda T_t\} = exp\{-t\psi(\lambda)\}.$$
 Then
the characteristic exponent $\Phi$ of the subordinate Brownian motion $X$ takes on a very simple form
$\Phi(x) =\psi(|x|^2)$ (our Brownian motion B runs at twice the usual speed). Therefore, properties of
$X$ should follow from properties of the Laplace exponent of the subordinator.

A lot of progress has been made in recent years in the study of the potential theory of subordinate
Brownian motions, see, for instance \cite{  RSV,  SV, KSV2, BBKRSV,  KSV3,  KSV}. At first,
the focus was on the potential theory of the process $X$ in the whole of $\Rd$, and basic results about the behaviour of the potential kernel and L\'{e}vy measures   were established for many particular examples of the subordinators including geometric stable (see  \cite{RSV, SV}). Then in  a natural path of  investigation    the
(killed) subordinate Brownian motion in an open subset was explored.   In the last few
years significant progress has been made in studying the potential theory of subordinate Brownian
motion killed upon exiting an open subset of $\Rd$ (see the survey \cite{KSV}). The main results include the Harnack inequality,  the boundary Harnack
principle and sharp Green function estimates.  However, such results were confined to the subordinated Brownian motions obtained by using $\psi$ not only being a complete Bernstein  function but also satisfying  certain property,
\begin{equation}\label{condition1}\psi(\lambda)\sim \lambda^{\beta}  l(\lambda), \lambda \to \infty,\end{equation}
where  $0<\beta<1$, and $l$ is a slowly varying function at $\infty$. Moreover, an extra assumption was set on $\psi$ to avoid a situation when the process $X$ is recurrent. In a recent paper \cite{KSV} the condition (\ref{condition1}) was  relaxed to comparability at $\infty$.

A natural question  about the Harnack inequality,  the boundary Harnack
principle and sharp Green function estimates arises in  the case  $\beta=0$ and without the  transience assumption. In this note we do not attempt to investigate a general such case (i.e. $\beta=0$), but we rather consider an important particular process, that is a geometric $\alpha$-stable process on the real line. For this process the corresponding subordinator has the Laplace exponent $\psi(\lambda)=\log(1+\lambda^{\alpha/2})$, $0<\alpha\le 2$. For $\alpha =2$ it is also called the  gamma variance process. The geometric $\alpha$-stable processes have been treated in the literature and play important role in the theory and applications (see e.g. \cite{MR}). Some  potential theory of them were  established in \cite{SV}, but to the best of our knowledge none sharp estimates  of the Green functions and Poisson kernels of open subsets, even in the one-dimensional case, are known.

Our main results are sharp estimates of the Green functions and Poisson kernels of  intervals (including a half-line), scale invariant Harnack inequality and the boundary Harnack inequality for harmonic functions on intervals.   It is worth mentioning that   our estimates take into account the size of intervals and the constants  depend only on the characteristics of the process, when Green functions and Poisson kernels are regarded. For example, we show that Poisson kernels for  half-lines for $\alpha$-stable and geometric $\alpha$-stable processes  are of the same order provided the starting point and the  exit point are  away from the boundary, if $0<\alpha<2$. On the other hand for starting points and exit points close to the boundary we have the same type of behaviour of the Poisson kernels for all $0<\alpha \le 2$.

\section{Preliminaries}
\setcounter{equation}{0}

    Throughout the paper  by $c, c_1\,\dots$ we  denote
       nonnegative constants which may depend on other constant parameters only.
        The value of $c$ or $c_1\,\ldots$ may change from line to line in a chain
         of estimates. If we use $C$ or $C_1,\, \dots$ then they are fixed constants.

      The notion $p(u)\approx q(u),\ u \in A$ means that the ratio
      $p(u)/ q(u),\ u \in A$ is bounded from  below and above
      by positive constants which may depend on other constant parameters only but does not depend on the set $A$.

We present in this section some basic material regarding the
geometric stable process. For more  detailed
information, see \cite{SV}. For questions regarding the
Markov and the strong Markov properties, semigroup properties,
Schr\"{o}dinger operators and basic potential theory, the reader
is referred to \cite{ChZ} and \cite{BG}.

We first introduce an appropriate class of subordinating
processes. As mentioned in the Introduction the geometric $\alpha$-stable process is obtained by subordination of the Brownian motion with a subordinator having the Laplace exponent $\psi(\lambda)=\log(1+\lambda^{\alpha/2})$, $0<\alpha\leq2$. The resulting process has the L\'{e}vy-Khintchine exponent $\Psi(x) = \psi(|x|^2)= \log(1+|x|^\alpha)$. Another way of constructing the  geometric $\alpha$-stable process is to subordinate the rotational invariant $\alpha$-stable process with the Gamma subordinator.
 Let $g_t(u)=\Gamma(t)^{-1}e^{-u}u^{t-1}$, $u, t>0$, denote the density function
of the Gamma subordinator $T_t$, with the Laplace transform \begin{equation}\label{gammaLaplace}Ee^{-\lambda T_t}=e^{-t
\log(1+\lambda)}.\end{equation}

 Let $Y^{\alpha}_t$ be the isotropic $\alpha$-stable process in $\Rd$ with the characteristic
  function of the form
 \begin{equation} \label{brownian}
 E^{0}e^{i\xi \cdot Y_t} = e^{-t|\xi|^\alpha}\,.
 \end{equation}
 Assume that the processes $T_t$ and $Y_t$ are stochastically independent.
 Then the process $X_t^{\alpha}= Y^{\alpha}_{T_t}$ is called the
  geometric stable process.
  In the sequel we use the generic notation $X_t$ instead of
   $X_t^{\alpha}$. From (\ref{gammaLaplace}) and  (\ref{brownian}) it is clear that
  the characteristic
  function of $X_t$ is  of the form
 $$ 
 E^{0}e^{i\xi \cdot X_t } = e^{-t\log(1+|\xi|^{\alpha})}\,.
 $$
In the case $\alpha=2$, i.e. $Y^2_t$ is a Brownian motion running twice the usual speed, the corresponding process is the symmetric gamma variance  process.

  $X_t$ is a L\'evy process (i.e. homogeneous, with independent
 increments). We always assume that sample paths of the process $X_t$ are right-continuous
 and have left-hand limits ("cadlag"). Then $X_t$ is Markov and has the strong
 Markov property under the so-called standard filtration.


  The  geometric stable density   can now be computed in the following way:
  \begin{equation*} \label{reldensity0}
 p_t(x)=\int_0^\infty s_u(x) g_t(u) du,
 \end{equation*}
 where $s_u(x)=\frac1{(2\pi)^d }\int_{\mathbb{R}^d}e^{-ix\xi- u{|\xi|^\alpha}}d\xi$ is  the density of the isotropic $\alpha$-stable process, defined by \pref{brownian}.
 In general potential theory  a very important role is played by
potential  kernels, which are
defined as

$$U(x,y)=\int_0^\infty p_t(x-y)dt, \ x,y\in \Rd,$$
if the defining integral above is finite. For the geometric process the potential
kernel is well defined for $d>\alpha$ but contrary to the stable case
 it is not expressible as an elementary function.
Recall that for the isotropic $\alpha$-stable process the
potential kernel is equal to $C|x-y|^{\alpha-d}$ for $d>\alpha$,
where $C$ is an appropriate constant depending on $\alpha$ and $d$.
 Nevertheless the asymptotic behaviour  of the potential kernel was   established in   \cite{SV}:

 \begin{equation}\label{0-potential}
  U(x-y)\approx  \frac1 {|x-y|^{d}\log^2(1+|x-y|^{-\alpha/2})},\quad x,y\in \Rd.
  \end{equation}
Note that (\ref{0-potential})  suggests that the process globally
 behaves like a stable one since  its potential kernel is asymptotically equivalent to the stable process, when $|x-y|$ is large.

We also recall the form of the density function $\nu(x)$ of the
L{\'e}vy measure of the
 geometric stable process:
  \begin{equation*}
  \nu(x)=  \int_0^\infty s_u(x)\, u^{-1}e^{-u} du.\label{levymeasure0}
	\end{equation*}
	The behaviour of the L{\'e}vy measure was investigated in \cite{SV}. We recall that result for the $d$-dimensional case, however  we need them only for $d=1$ in the present paper.
	For $\alpha=2$  we have
	\begin{equation*}
	\nu(x)\approx \frac {1+|x|^{(d-1)/2}}{|x|^{d}}e^{-|x|},\label{levymeasure3}
	\end{equation*}
	and for $0<\alpha<2$
	\begin{equation*}
	\nu(x)\approx \frac1{|x|^d(1+|x|^{\alpha})}.\label{levymeasure}
	\end{equation*}
	For $d=1$, which is the case investigated in this paper, for $\alpha=2$ we even have an exact formula
  \begin{equation}
	\nu(x)=|x|^{-1}e^{-|x|}.\label{levymeasure2}
	\end{equation}


    The {\it first exit time} of an (open)
   set  $D\subset {\Rd}$
   by the process $X_t$ is defined by the formula
   $$
   \tau_{D}=\inf\{t> 0;\, X_t\notin D\}\,.
   $$
The fundamental object of the potential theory is the {\it killed process} $X_t^D$
  when exiting the set $D$. It is defined in terms of sample paths up to time $\tau_D$.
  More precisely, we have the following "change of variables" formula:
  $$
  E^x f(X_t^D) =  E^x[t<\tau_D; f(X_t)]\,,\quad t>0\,.
  $$
  The density function of the transition probability of the process $X_t^D$ is denoted
  by $p_t^{D}$. We have
  \begin{equation*}
  p_t^{D}(x,y) = p_t(x-y) -
   E^x[t> \tau_D; p_{ t-\tau_D}(X_{\tau_D}-y)]    \,, \quad x, y \in {\Rd}\,.\label{density100}
  \end{equation*}
  Obviously, we obtain
   $$
     p_t^{D}(x,y) \le p_t(x,y) \,, \quad x, y \in {\Rd}\,.
   $$

  $(p_t^{D})_{t>0}$ is a strongly contractive semigroup (under composition) and shares most
  of properties of the semigroup $ p_t$. In particular, it is strongly Feller and
  symmetric: $ p_t^{D}(x,y) =  p_t^{D}(y,x)$.

  The potential kernel of the process $X_t^D$ is called the
  {\it Green function} of the set $D$ and is denoted by $G_D$. Thus, we have
  $$
   G_D(x,y)= \int_0^{\infty} p_t^{D}(x,y)\,dt\,.
  $$

  Another important  object in the potential theory of $X_t$ is the
  {\it harmonic measure}  of the
  set $D$. It is defined by the formula:
  $$ 
  P_D(x,A)=
  E^x[\tau_D<\infty; {\bf{1}}_A(X_{\tau_D})].
  $$
  The density kernel (with respect to the Lebesgue measure) of  the measure $P_D(x,A)$ (if it exists) is called the
  {\it Poisson kernel} of the set $D$.
  The relationship between the Green function of $D$ and the harmonic measure is provided by the Ikeda-Watanabe formula \cite{IW},
$$ 
  P_D(x,A)= \int_A \int_D G_D(x,y)\nu(y-z)dy dz, \quad A\subset (\bar{D})^c.
  $$
In the case which we investigate in this paper, that is when $D$ is an open interval or a half-line, the above formula holds for any Borel $A\in D^c$.


Now we define harmonic and regular harmonic functions. Let $u$ be
a Borel measurable function on $\Rd$. We say that $u$ is {\em
harmonic} function in an open set $D\subset \Rd$ if
$$u(x)=E^xu(X_{\tau_B}), \quad x\in B,$$
for every bounded open set $B$ with the closure
$\overline{B}\subset D$. We say that $u$ is {\em regular harmonic}
if
$$u(x)=E^x[\tau_D<\infty; u(X_{\tau_D}))], \quad x\in D.$$


The following lemma provides a very useful lower bound for the Green function. Its proof
closely follows the approach used in  \cite{RSV}, where the bounds
on the potential kernels (Green functions for the whole $\Rd$)
were established for some special subordinated  Brownian motions
(in particular for  our process for $d >\alpha $). We omit the proof, since one can proceed exactly in the same way as in the proof of Lemma 2.11 in \cite{GR3}.
\begin{lem}    \label{potential_lower} For any open set $D\in {\Rd}$  we have
$$  G_D(x,y)\ge   G_{D}^{(\alpha)}(x,y),$$
where  $G_{D}^{(\alpha)}(x,y)$ is the Green function of $D$ for the isotropic
$\alpha$-stable process.
 \end{lem}

\section{Properties of the  exit time from interval}

Now, we briefly recall the basic notions of the fluctuation theory for L{\'e}vy processes. For the general account of this theory we refer the reader to \cite{Doney}.  Suppose that $X_t$ is a general one-dimensional  L{\'e}vy process.  Let $L_t$ be the local time of the process $X_t$ reflected at its supremum $M_t=sup_{s\le t}X_s$, and denote by $L^{-1}_s$ the right-continuous inverse of $L_t$, the ascending ladder time process for $X_t$. This is a (possibly killed) subordinator, and $H_s = X(L^{-1}_s) = M(L^{-1}_s)$ is another (possibly killed) subordinator, called the ascending ladder-height process. The Laplace exponent of the increasing ladder process, that is, the (possibly killed) bivariate subordinator $(L^{-1}_s, H_s)$ ($s < L(\infty)$), is denoted by $\kappa(z, \xi)$,
$$ \kappa(z,\xi) = c\exp\left(\int_0^\infty \int_{[0,\infty)}(e^{-t}-e^{-zt-\xi x})t^{-1}P(X_t\in dx )dt\right),
$$
where $c$ is a normalization constant of the  local time. Since our results are not affected by the choice of $c$ we assume that $c = 1$.

Moreover, if $X_t$ is not a compound Poisson process, then by~\cite{bib:f74}, Corollary~9.7,
$$
 \kappa(0, \xi)  = \exp\expr{\frac{1}{\pi} \int_0^\infty \frac{\xi \log \Psi(\zeta)}{\xi^2 + \zeta^2} \, d\zeta} = \psi^\dagger(\xi) ,
$$
where $\Psi(\xi)$ is the L\'{e}vy exponent of $X_t$.
By  $V(x) = \int_0^{\infty}\pr(H_s \le x)ds$ we denote  the renewal function of the process $H_s$. It is well known that $V$ is subadditive, that is
$$V(x+y)\le V(x)+V(y), \quad x,y\ge 0.$$

The behaviour of the renewal function and its derivative is crucial for our purposes.   The following result was obtained independently in \cite{KMR} and \cite{KSV}. In \cite{KSV} the assumption on the process $X_t$ was a bit more restrictive.
\begin{prop}\label{prop:regvar}
Let $X_t$ be a symmetric L\'{e}vy process such that    its L\'{e}vy-Khintchine exponent  $\Psi(\theta)$  has the property that $\Psi(\theta)$,   $\theta^2/\Psi(\theta)$ are increasing functions.  Then
$$\psi^\dagger(\xi)\approx \sqrt{\Psi(\xi)}$$ and
 $$V(x)\approx \frac  1{\sqrt{\Psi(1/x)}}.$$
Suppose additionally that
  $\Psi(\theta)$ is  regularly varying at zero (at $\infty$) with positive exponent and $V'(x)$ is eventually monotone at infinity (at zero). Then,
  $$V'(x)\approx \frac  1{x\sqrt{\Psi(1/x)}}, \quad x\to \infty\quad (x\to 0).$$

\end{prop}
In the case when  $\Psi(\theta)$ is  slowly varying at $\infty$ or $0$ the above proposition is of little help in estimating $V'(x)$ and we need to use another tool.  We will take advantage of the following result proved recently
  in \cite{KMR}.
\begin{prop}
\label{prop:v}
Let $\Psi(\xi)$ be the L{\'e}vy-Khintchine exponent of a symmetric L{\'e}vy process $X_t$, which is not a compound Poisson process, and suppose that $\Psi(\xi) = \psi(\xi^2)$ for a complete Bernstein function $\psi$. Then $V$ is a Bernstein function, and
\begin{equation}
\label{eq:v}
 V(x)  = b x + \frac{1}{\pi} \int_{0^+}^\infty \im \expr{-\frac{1}{\psi^+(-\xi^2)}} \frac{\psi^\dagger(\xi)}{\xi} \, (1 - e^{-x \xi}) d\xi ,\;  x > 0 . \end{equation}
Moreover, $V'$ is a completely monotone function and
\begin{equation}
\label{eq:vprime}
 V'(x) = b + \frac{1}{\pi} \int_{0^+}^\infty \im \expr{-\frac{1}{\psi^+(-\xi^2)}} \psi^\dagger(\xi) e^{-x \xi} d\xi ,\;  x > 0 ,
\end{equation}
where $b = \lim_{\xi \to 0^+} (\xi / \sqrt{\psi(\xi^2)})$.

\end{prop}

Here the expression $\im (-1 / \psi^+(-\xi^2)) d\xi$ should be understood in the distributional sense, as a weak limit of measures $\im (-1 / \psi(-\xi^2 + i \eps)) d\zeta$ on $\xi \in (0, \infty)$ as $\eps \to 0^+$. The measure $\im (-1 / \psi^+(-\xi^2)) d\xi$ has an atom of mass $\pi b$ at $0$, and this atom is not included in the integrals from $0^+$ to $\infty$ in~\eqref{eq:v} and~\eqref{eq:vprime}.


For the rest of  this section we assume that $ X_t$ is a symmetric L\'{e}vy process which is not a compound Poisson process with its the renewal function $V$ corresponding to such a choice of the local time that the Laplace exponent of the ladder time process is $\kappa(z,0)= \sqrt{z}$.
We start with an estimate of the distribution function of the exit time $\tau$ from a half-line $(0, \infty)$ which was obtained in \cite{KMR} (Corollary 2).

 \begin{lem} \label{tau} Let $\tau$ be the exit time from $(0,\infty)$. There is an absolute constant $C_1$ such that

\begin{equation}\label{time_estimate}P^x(\tau>t)\ge   C_1\left(1\wedge  \frac{V(x)}{\sqrt{t}}\right), \quad x, t>0.\end{equation}
\end{lem}

\begin{lem} \label{exit_upper}

Let $0<x<R$  and $\tau_{(0, R)}$ be the exit time from the interval  $(0, R)$. Then
\begin{equation}\label{exit_estimate}
 P^x(\tau_{(0, R)}<\tau)\le \frac {V(x)}{V(R)}.\end{equation}
\end{lem}
\begin{proof}
The inequality (\ref{exit_estimate}) was observed in \cite{KSV3} for the case when the resolvent kernels  of the L\'{e}vy process are absolutely continuous with respect to the Lebesgue measure and   $0$ is regular for $(0,\infty)$. Let $Y^\epsilon_t= X_t+\epsilon B_t$, where $B_t$ is a Brownian motion independent of $X_t$. Obviously $Y^\epsilon$ satisfies the above conditions.  Furthermore  it is easy to see that the renewal function of the $Y^\epsilon$ converges pointwise  to  $V$. Moreover, since the  process $Y^\epsilon$ converges a.s. to $X$, uniformly on  bounded intervals,  the result follows by the limiting argument.

\end{proof}

\begin{prop} \label{exptime}

Let $0<x<R$. Then

$$\frac {C_1^4}{16} V(x\wedge(R-x)){V}(R)\le  E^x\tau_{(0, R)}\le V(x\wedge(R-x)){V}(R),$$
where $C_1$ is the constant from Lemma \ref{tau}.
\end{prop}
\begin{proof} From symmetry it is enough to consider $x\le R/2$. According to \cite{Bertoin} (page 176), Theorem 20, for any measurable non-negative  function $f:[0,\infty)\mapsto [0,\infty)$, we have
$$E^x\left[\int_0^\tau f(X_t) dt\right]= \int_{[0,\infty)}{V}(dy)\int_{[0,x]}V(dz)f(x+y-z).$$
We take $f=I_{[0,R]}$. Then

\begin{eqnarray*}E^x\tau_{(0, R)}&=&E^x\left[\int_0^{\tau_{(0,R)}} f(X_t) dt\right] \le  E^x\left[\int_0^\tau f(X_t) dt)\right]= \int_{[0,\infty)}{V}(dy)\int_{[0,x]}V(dz)f(x+y-z)\\
&\le& \int_{[0,R]}{V}(dy)\int_{[0,x]}V(dz)= {V}(R)V(x),\end{eqnarray*}
which completes the proof of the upper bound.

To prove the lower bound we observe that
$$P^x(\tau>t)\le P^x(\tau_{(0, R)}>t)+P^x(\tau_{(0, R)}<\tau)\le \frac{E^x\tau_{(0,R)}}{t}+ P^x(\tau_{(0, R)}<\tau). $$
Hence from (\ref{time_estimate}), for $\sqrt{t}>V(x)$, and from  (\ref{exit_estimate}) we obtain
\begin{eqnarray*}
E^x\tau_{(0,R)}&\geq& t(P^x(\tau>t)-P^x(\tau_{(0, R)}<\tau))\geq t\left(C_1\frac{V(x)}{\sqrt{t}}-\frac{V(x)}{V(R)}\right)\\
&=&V(x)\sqrt{t}\left(C_1-\frac{\sqrt{t}}{V(R)}\right).
\end{eqnarray*}
Let $\sqrt{t}=\frac{C_1}{2}V(R)$ then, for $2V(x)\le C_1V(R)$, we have
$$E^x\tau_{(0,R)}\geq \frac{C_1^2}{4}V(x)V(R).$$
Next, we deal with   $(C_1/2)V(R)\le V(x)\le V(R/2)$. Then for $x_0:\, V(x_0)=(C_1/2)V(R)$, using the already proved lower bound, we obtain
$$E^R\tau_{(0,2R)}\geq E^{x_0}\tau_{(0,R)}\geq   \frac{C_1^2}{4}V(x_0)V(R)= \frac{C_1^3}{8}V(R)V(R). $$
Finally,  let $x_0\le x\leq R/2$. Then $V(x)\ge V(x_0)=(C_1/2)V(R)$ which implies

$$E^{x}\tau_{(0,R)}\geq E^{x}\tau_{(0,2x)}\geq   \frac{C_1^3}{8}V(x)V(x)\ge \frac{C_1^4}{16}V(x)V(R). $$
\end{proof}


\begin{rem}\label{GreenExit}
Assume that the Green function of the half-line exists. Then, for $x\le R/2$,
$$\int^R_0G_{(0,\infty)}(x,y)dy\leq V(x)V(R).$$
\end{rem}

Another consequence of Lemma \ref{tau} is the following two sided bound on the exit probability, which is interesting on its own. There is a huge literature on the subject of so called {\em scale} functions which describe the probability that the process leaves a given interval through its right end. This function has been found for numerous examples of spectrally negative processes (see a survey \cite{KKR} and  references therein). To the best of our knowledge, for symmetric processes, except the Brownian motion or the $\alpha$-stable motions exact formulas are not known, hence optimal estimates seem important.
\begin{prop}

Let $0<x<R$  and $\tau_{(0, R)}$ be the exit time from the interval  $(0, R)$. Then
\begin{equation*}\label{exit_estimate_twoside}
\frac{C_1^2}4 \frac {V(x)}{V(R)}\le P^x(\tau_{(0, R)}<\tau)\le \frac {V(x)}{V(R)},\end{equation*}
where $C_1$ is the constant from Lemma \ref{tau}.
\end{prop}

\begin{proof} We deal only with the lower  bound.
From  Proposition  \ref{exptime} we infer
$$ P^x(\tau_{(0, R)}>t)\le \frac{ E^x\tau_{(0, R)}}t\le \frac{V(x)V(R)}t.$$
Next, observe that
$$P^x(\tau>t)\le P^x(\tau_{(0, R)}>t)+P^x(\tau_{(0, R)}<\tau)\le \frac {V(x)V(R)}t+ P^x(\tau_{(0, R)}<\tau).$$
Hence, form (\ref{time_estimate}), for  $\sqrt{t}\ge V(x)$, we have

$$  C_1\frac{V(x)}{\sqrt{t}}- \frac {V(x)V(R)}t\le P^x(\tau_{(0, R)}<\tau).$$
If we choose  $\sqrt{t}= \frac{2}{C_1} V(R)\ge\frac{2}{C_1}  V(x)\ge  V(x) $   then

$$  C_1\frac{V(x)}{\sqrt{t}}- \frac {V(x)V(R)}t= \frac {C_1^2}4\frac{V(x)}{V(R)}.$$
This yields

$$P^x(\tau_{(0, R)}<\tau)\ge \frac{C_1^2}4 \frac {V(x)}{V(R)},\quad x<R. $$

\end{proof}





\section{Green function and Poisson kernel of the half-line}
\setcounter{equation}{0}

From now on, we assume that $X_t$ is the one-dimensional geometric stable process. In order to find precise estimates of the Green function and the Poisson kernel we need to have nice estimates of the renewal function and its derivative of the
ladder height process of $X_t$. Note that the Laplace exponent  $\phi(\lambda)=\log(1+\lambda^{\alpha/2})$ is a complete Bernstein function, therefore we can use Proposition \ref{prop:v}. It is well known (see e.g. \cite{BBKRSV}) that the  derivative $V^{'}(x)$ of the renewal function  is decreasing. Monotonicity of $V'$  together with subadditivity of $V$ is frequently used in the sequel.

\begin{lem}\label{potential}Let $\alpha\in(0,2]$ and $x>0$ then
$$V^{'}(x)\approx\frac{1}{x\log^{3/2}\left(1+x^{-\alpha/3}\right)} $$
and
$$V(x)\approx\frac{1}{\log^{1/2}\left(1+x^{-\alpha}\right)} .$$
\end{lem}
\begin{proof}
The estimates of the renewal function $V(x)$, $x>0,$ as well as its derivative $V^{\prime}(x)$ for $x>1$
 follow from Proposition  \ref{prop:regvar}.
To deal with $V{'}(x)$, for $x\le 1$, we apply Proposition \ref{prop:v}
with $\psi(\xi)=\log(1+\xi^{\alpha/2})$. Then it is evident that   $b=1$ for $\alpha=2$ and $b=0$ otherwise. Moreover,
\begin{equation*}\im\left(-\frac{1}{\psi^{+}(-\xi^2)}\right)= \left\{%
\begin{array}{ll}
    \frac{\pi}{\pi^2+\log^2(\xi^2-1)}\textbf{1}_{(1,\infty)}(\xi), & \hbox{$\alpha=2$,} \\
    \frac{Arg(z)}{Arg(z)^2+\frac{1}{4}\log^2(1+2\xi^\alpha\cos\frac{\alpha\pi}{2}+\xi^{2\alpha})}, & \hbox{$\alpha<2$,} \\
    \end{array}%
\right.\end{equation*}
where $z=1+\xi^\alpha\cos\frac{\alpha\pi}{2}+i\xi^\alpha\sin\frac{\alpha\pi}{2}$. Next, by Proposition  \ref{prop:regvar},   $\psi^\dagger(\xi)\approx \sqrt{\psi(\xi^2)}$ .
Let $\mu(\xi)=\psi^\dagger(\xi)Im\left(-\frac{1}{\psi^+(-\xi^2)}\right)$.
Note that $\mu(\xi)\approx \log^{-3/2}\xi$ for $\xi\geq2$.

If $x\leq 1$ 
then by (\ref{eq:vprime}) we have
\begin{eqnarray*}
V'(x)&=&b+\int^2_0e^{-x\xi}\mu(\xi)d\xi+\int^\infty_2e^{-x\xi}\mu(\xi)d\xi\approx 1+\int^\infty_2\frac{e^{-x\xi}}{\log^{3/2}\xi}d\xi\\&\approx&\int^\infty_{4x^{-1}}\frac{e^{-x\xi}}{\log^{3/2}\xi}d\xi+\int^{4x^{-1}}_2\frac{e^{-x\xi}}{\log^{3/2}\xi}d\xi\\
&\approx&\int^{4x^{-1}}_{2}\frac{1}{\log^{3/2}\xi}d\xi\approx \frac{1}{x\log^{3/2}(4x^{-1})}.
\end{eqnarray*}
In the last line we use the inequality $$0<\int^\infty_{4x^{-1}}\frac{e^{-x\xi}}{\log^{3/2}\xi}d\xi< \frac{e^{-4}}{x\log^{3/2}(4x^{-1})}.$$

\end{proof}
As an immediate consequence we obtain the following estimate \begin{equation}\label{V'approx}V'(x)\approx\frac{V(x)}{x\log(2+x^{-1})}, \ x>0.\end{equation}

The next lemma provides useful estimates for some integrals involving $V$ used in the sequel.
\begin{lem}We have the following estimates
\begin{eqnarray}
\int^x_0V(y)dy&\approx &xV(x),\quad x>0,\label{IntV1}\\
\int^x_1V(y)\frac{dy}{y}&\approx&V(x),\quad x\geq2,\label{IntV2}\\
\int^1_xV(y)\frac{dy}{y}&\approx&\frac{1}{V(x)},\quad x\leq1/2,\label{IntV3}\\
\int^1_xV^\beta(y)\frac{dy}{y^2}&\approx&\frac{V^\beta(x)}{x},\quad x\leq1/2,\;\beta>0,\label{IntV4}
\end{eqnarray}
\end{lem}
\begin{proof}
The first approximation is true for all L\'{e}vy processes. Indeed, by monotonicity and subadditivity  of $V$ we have
$$\frac{1}{4}xV(x)\leq \frac{x}{2}V\left(\frac{x}{2}\right)\leq \int^x_{x/2}V(y)dy\leq\int^x_{0}V(y)dy\leq xV(x).$$
For $y\geq 1$, by Lemma \ref{potential}, we get $\frac{V(y)}{y}\approx V'(y)$, which leads to  (\ref{IntV2}). Next,
Lemma \ref{potential} implies
$$\int^1_xV(y)\frac{dy}{y}\approx\int^1_x\frac{dy}{y\log^{1/2}\left(1+\frac{1}{y}\right)}\approx\log^{1/2}\left(1+\frac{1}{x}\right)\approx\frac{1}{V(x)}.$$
Moreover, by the (\ref{V'approx}),
$$\lim_{x\rightarrow 0^+}\frac{\frac{V^\beta(x)}{x^2}}{\frac{V^\beta(x)}{x^2}-\frac{\beta V^{\beta}(x)V'(x)}{V(x)x}}=1,$$
which yields (\ref{IntV4})  by applying the l'Hospital's rule.
\end{proof}

By Theorem 20 (page 176) of \cite{Bertoin} we have a basic and very  useful formula for the Green function of the half-line.
\begin{lem}\label{Greenformula} For $0<x<y$ we have
$$G_{(0,\infty)}(x,y)=\int^x_0V^{'} (u)V^{'} (y-x+u)du.$$

\end{lem}

At this point let us recall that the exact formulas for the
Brownian Green  functions are well known for several regular sets
as intervals or half-lines (see e.g. \cite{Ba}). Since some of them
will be useful in the sequel we  list them for the future
reference. Recall that the Brownian motion we refer to in this
paper  has its clock running twice faster then the usual Brownian
motion. Denote the renewal function  for the symmetric  $\alpha$-stable process (properly normalized)  by $V^{(\alpha)}(x)=x^{\alpha/2}$, where $\alpha \in (0,2]$. For the half-line we have
 \begin{equation*}\label{gaussGreen1}G^{(2)}_{(0,\infty)}(x,y)= x\wedge y=V^{(2)}(x\wedge y), \quad x,y>0, \end{equation*}
while for  the finite interval $(0, R)$,
 \begin{equation*}\label{gaussGreen3}G^{(2)}_{(0, R)}(x,y)= \frac {x(R-y) \wedge y(R-x)}R=\frac{(V^{(2)}(x)V^{(2)}(R-y))\wedge (V^{(2)}(R-x)V^{(2)}(y))}{R}, \quad x, y\in (0,R). \end{equation*}

We also  recall known estimates for stable case (see e.g. \cite{BB}),
\begin{equation*}\label{greenstable}
G_{(0,R)}^{(\alpha)}(x,y)\approx \left\{%
\begin{array}{ll}
    \min\left\{\frac{1}{|x-y|^{1-\alpha}},
    \frac{(\delta_R(x)\delta_R(y))^{\alpha/2}}{|x-y|}\right\}=\min\left\{\frac{1}{|x-y|^{1-\alpha}},
    \frac{V^{(\alpha)}(\delta_R(x))V^{(\alpha)}(\delta_R(y))}{|x-y|}\right\}, & \hbox{$\alpha<1$,} \\
    \ln\left(1+\frac{(\delta_R(x)\delta_R(y))^{1/2}}{|x-y|}\right)=\ln\left(1+\frac{V^{(1)}(\delta_R(x))V^{(1)}(\delta_R(y))}{|x-y|}\right), & \hbox{$\alpha=1$,} \\
    (\delta_R(x)\delta_R(y))^{\frac{\alpha-1}{2}}\wedge
    \frac{(\delta_R(x)\delta_R(y))^{\alpha/2}}{|x-y|}=\frac{V^{(\alpha)}(\delta_R(x))V^{(\alpha)}(\delta_R(y))}{(\delta_R(x)\delta_R(y))^{1/2}}\wedge
    \frac{V^{(\alpha)}(\delta_R(x))V^{(\alpha)}(\delta_R(y))}{|x-y|}, & \hbox{$\alpha>1$,} \\
\end{array}%
\right.
\end{equation*}
where $\delta_R(x)=x\wedge(R-x)$, for $R<\infty$ and $\delta_{\infty}(x)=x$.

Define a function $\hat{G}_{(0,\infty)}^{(\alpha)}(x,y)$ such that
\begin{equation*}
\hat{G}_{(0,\infty)}^{(\alpha)}(x,y)=\left\{
\begin{array}{ll}V(x\wedge y), & \hbox{$\alpha=2$,} \\
    \min\left\{(xy)^{(\alpha-1)/2},
    \frac{V(x)V(y)}{|x-y|}\right\}, & \hbox{$1<\alpha<2$,} \\
		\ln\left(1+\frac{V(x)V(y)}{V^2(|x-y|)}\right), & \hbox{$\alpha=1$,} \\
		\min\left\{1,
    \frac{V(x) V(y)}{V^2(|x-y|)}\right\}\frac{1}{|x-y|^{1-\alpha}}, & \hbox{$\alpha<1$,} \\
    \end{array}%
\right.\end{equation*}
Note that  \begin{equation}\label{GcomphatG}\hat{G}_{(0,\infty)}^{(\alpha)}(x,y)\approx G_{(0,\infty)}^{(\alpha)}(x,y),\end{equation}
if $\alpha\neq 1$ for $x,y>1/2$ and if $\alpha=1$ for $x,y>1/2$ and $|x-y|>1/2$.
Now, we are at the  position to prove the optimal estimates of the Green function of $(0, \infty)$, which are crucial for the rest of the paper.
\begin{thm}\label{Greenhalf}Let $0<x,y$. Then
$$G_{(0,\infty)}(x,y)\approx \left(1\wedge\frac{V(x)V(y)}{V^2(|y-x|)}\right)\frac{1}{|y-x|\log^{2}(2+|y-x|^{-1})}+\hat{G}_{(0,\infty)}^{(\alpha)}(x,y).$$
\end{thm}
\begin{proof}
Note that by monotonicity $V'$ and Lemma \ref{potential} we have for $0< u\leq w$,
\begin{equation}\label{doubleV'}V^{'}(2w)\leq V^{'}(w+u)\leq V'(w)\approx  V^{'}(2w).\end{equation}
Assume that $0<x<y$. We split the proof into several cases.

{\em Case 1:}  $2x\leq y$.

 In this region $y/2\leq y-x< y$ so,  by subadditivity of $V$,  $\frac{V(x)V(y)}{V^2(|y-x|)}\leq 4 $.
Hence, by Lemma \ref{Greenformula} and (\ref{doubleV'}) it follows
\begin{eqnarray}\label{case1}
G_{(0,\infty)}(x,y)&\approx&\int^x_0V'(u)V'(y)du=V(x)V^{'}(y)\approx \frac{V(x)V(y)}{|y-x|}\frac{1}{\log(2+|y-x|^{-1})}.
\end{eqnarray}

For $|y-x|\leq 1$, by Lemma \ref{potential} we get  $ V^2(|y-x|)\approx \log^{-1}(2+|y-x|^{-1})$, which leads to
\begin{eqnarray*}G_{(0,\infty)}(x,y)&\approx& \frac{V(x)V(y)}{V^2(|y-x|)}\frac{1}{|y-x|\log^2(2+|y-x|^{-1})}\\
&\approx&\frac{V(x)V(y)}{V^2(|y-x|)}\frac{1}{|y-x|\log^2(2+|y-x|^{-1})}+\hat{G}_{(0,\infty)}^{(\alpha)}(x,y),
\end{eqnarray*}
where the last step follows from the inequality
 $ \frac{V(x)V(y)}{|y-x|}\frac{1}{\log(2+|y-x|^{-1})}\ge \hat{G}_{(0,\infty)}^{(\alpha)}(x,y)$.

Next, for $y-x>1$,  $\hat{G}^{(\alpha)}_{(0,\infty)}(x,y)\approx \frac{V(x)V(y)}{y-x}$. Again, by Lemma \ref{potential}, we have  $V^2(y-x)\log^2(2+|y-x|^{-1})\approx (y-x)^\alpha$, for $|y-x|>1$. Hence,
$$G_{(0,\infty)}(x,y)\approx \frac{V(x)V(y)}{|y-x|}\approx \hat{G}_{(0,\infty)}^{(\alpha)}(x,y)\approx \frac{V(x)V(y)}{V^2(|y-x|)}\frac{1}{|y-x|\log^2(2+|y-x|^{-1})}+\hat{G}_{(0,\infty)}^{(\alpha)}(x,y).$$

{\em Case 2:} $x+1/2<y<2x$.

 Note that $x>1/2$. By (\ref{GcomphatG}), $G^{(\alpha)}_{(0,\infty)}(x,y)\approx \hat{G}^{(\alpha)}_{(0,\infty)}(x,y)$ and $G^{(\alpha)}_{(0,\infty)}(x,y)\geq c \frac{1}{|y-x|\log^2(2+|y-x|^{-1})}$. By Lemma \ref{Greenformula} and  (\ref{doubleV'}),
\begin{eqnarray*}
G_{(0,\infty)}(x,y)&=&\int^{1/2}_0V'(u)V'(y-x+u)du+\int^x_{1/2}V'(u)V'(y-x+u)du\\&\approx&  V'(y-x)V(1/2)+\int^{1/2}_0V^{'}(u)V'(y-x+u)du.
\end{eqnarray*}
Similarly, \begin{eqnarray*}
G^{(\alpha)}_{(0,\infty)}(x,y)&\approx&  (V^{(\alpha)})'(y-x)V^{(\alpha)}(1/2)+\int^{x}_{1/2}(V^{(\alpha)})'(u)(V^{(\alpha)})'(y-x+u)du.
\end{eqnarray*}
 It follows from Lemma \ref{potential} that $V'(u)\approx (V^{(\alpha)})'(u)$ and $V(u)\approx V^{(\alpha)}(u)$ for $u\geq1/2$. Hence
$$G_{(0,\infty)}(x,y)\approx G^{(\alpha)}_{(0,\infty)}(x,y)\approx \hat{G}^{(\alpha)}_{(0,\infty)}(x,y)+\frac{1}{|y-x|\log^2(2+|y-x|^{-1})}.$$

{\em Case 3:} $x<y<(x+1/2)\wedge 2x$.

 We use Lemma \ref{Greenformula} and (\ref{doubleV'}) to get
\begin{eqnarray*}
G_{(0,\infty)}(x,y)&=& \int^{y-x}_0V^{'}(u)V^{'}(y-x+u)du+\int^x_{y-x}V^{'}(u)V^{'}(y-x+u)du\\
&\approx& V^{'}(y-x)V(y-x)+\int^x_{y-x}V^{'}(u)V^{'}(u)du.
\end{eqnarray*}
By Lemma \ref{potential} the first term is estimated  in the following way
\begin{equation}\label{V'V}
V^{'}(y-x)V(y-x)\approx \frac{1}{|y-x|\log^2(2+|y-x|^{-1})}.
\end{equation}
It remains to estimate  $\int^{x}_{y-x}(V^{'})^2(u)du$.
Note that, by Lemma \ref{potential}, $V^{'}(u)\approx \frac {V^3(u)}u, u\le 2$. Hence, for $x\leq2$, by  (\ref{IntV4}),
\begin{eqnarray}\label{intV'1}
\int^x_{y-x}(V^{'})^2(u)du&\leq&c\int^{2}_{y-x}\frac{V^6(u)}{u^2}du\approx  \frac{V^6(y-x)}{y-x}\approx \frac{1}{|y-x|\log^3(2+|y-x|^{-1})} .
\end{eqnarray}
For $x>2$, again by Lemma \ref{potential},
\begin{eqnarray*}
\int^x_{1}(V^{'})^2(u)du&\approx&\left\{
\begin{array}{ll}x^{\alpha-1}, & \hbox{$\alpha>1$,} \\
    \log x, & \hbox{$\alpha=1$,} \\
		1, & \hbox{$\alpha<1$,} \\
    \end{array}%
\right.\approx \left\{
\begin{array}{ll}(xy)^{(\alpha-1)/2}, & \hbox{$\alpha>1$,} \\
    \log (1+x^{1/2}y^{1/2}), & \hbox{$\alpha=1$,} \\
		1, & \hbox{$\alpha<1$.} \\
    \end{array}%
\right.
\end{eqnarray*}
Hence, for $\alpha>1$,
\begin{equation}\label{intV'4}
\int^x_{1}(V^{'})^2(u)du\approx \hat{G}^{(\alpha)}_{(0,\infty)}(x,y),
\end{equation}
and, for $\alpha\leq 1$,
\begin{equation}\label{intV'2}
\int^x_{1}(V^{'})^2(u)du\leq c \hat{G}^{(\alpha)}_{(0,\infty)}(x,y).
\end{equation}
Moreover, by (\ref{intV'1}),
\begin{equation}\label{intV'3}
\int^x_{y-x}(V^{'})^2(u)du\geq c  \hat{G}^{(\alpha)}_{(0,\infty)}(x,y).
\end{equation}
Finally, combining (\ref{V'V})-(\ref{intV'3}) 
we get
$$G_{(0,\infty)}(x,y)\approx \frac{1}{|y-x|\log^2(2+|y-x|^{-1})}+ \hat{G}^{(\alpha)}_{(0,\infty)}(x,y).$$

\end{proof}

\begin{rem}\label{Greenhalflineasymp} Let $|x-y|>A$. Then there exists a constant $C=C(A)$ such that
\begin{equation*}
C^{-1} \hat{G}^{(\alpha)}_{(0,\infty)}(x,y)\leq G_{(0,\infty)}(x,y)\leq C \hat{G}^{(\alpha)}_{(0,\infty)}(x,y).
\end{equation*}
Moreover, if $x,y<4$ then
\begin{equation*}\label{boundary}
G_{(0,\infty)}(x,y)\approx \left(1\wedge\frac{V(x)V(y)}{V^2(|y-x|)}\right)|y-x|^{-1}\log^{-2}(1+|y-x|^{-1}).
\end{equation*}
\end{rem}

In the rest of this section we prove the estimates of the Poisson kernel of $(0,\infty)$.
Recall that for $0<\alpha<2$ we know the form of the Poisson kernel for the $\alpha$-stable process (see e.g. \cite{BGR2}),
$${P}^{(\alpha)}_{(0,\infty)}(x,z)=C_\alpha\frac{V^{(\alpha)}(x)}{V^{(\alpha)}(|z|)}\frac{1}{x-z}, \quad z<0<x.$$

\begin{lem}\label{PoissonHat}Let $z<0<x$. Assume that $x\vee |z|\geq 1$, then we have
$$P_{(0,\infty)}(x,z)\approx\left\{%
\begin{array}{ll}e^z\frac{V(x\wedge 1)}{V(|z|)}, & \hbox{$\alpha=2$,} \\
    \frac{V(x)}{V(|z|)}\frac{1}{x-z}, & \hbox{$\alpha<2$.} \\
    \end{array}%
\right.
$$
\end{lem}

\begin{proof}

{\em Case 1:} $\alpha=2$ and $ z<-1$, $x>0$.

Observe that $$V(x\wedge1) V(y\wedge1)\le V(x\wedge y )V(1)\le  V(x\wedge1)V(y\vee1).$$
Assume that $z\le -1$. Since by, Remark \ref{Greenhalflineasymp}, $G_{(0,\infty)}(x,y)\approx   V(x\wedge y ), y\ge x+1$, then using formula (\ref{levymeasure2}) and (\ref{GreenExit}),
\begin{eqnarray*}
P_{(0,\infty)}(x,z)&\le&c\int^\infty_0V(x\wedge y)\frac{e^{z-y}}{y-z}dy +c\int^{x+1}_{0\vee(x-1)}G_{(0,\infty)}(x,y)\frac{e^{z-y}}{y-z}dy\\
&\le &\frac{ce^z}{|z|V(1)} V(x\wedge 1) \int^\infty_0V(y\vee1)e^{-y} dy+ c\frac{e^{z-x}}{x-z}\int^{2(x+1)}_{0}G_{(0,\infty)}(x,y)dy \\
&\le&
 cV(x\wedge 1)\frac{e^z}{V(|z|)} + cV(x) V(2x+2) e^{-x}\frac{e^z}{V(|z|)}\\
 &\le& c  V(x\wedge 1)\frac{e^z}{V(|z|)}
\end{eqnarray*}
Similarly
\begin{eqnarray*}
P_{(0,\infty)}(x,z)&\ge &c\int^1_0V(x\wedge y)\frac{e^{z-y}}{y-z}dy \\
&\ge &\frac{ce^z}{-2zV(1)} V(x\wedge 1) \int^1_0V(y)e^{-y} dy\approx  V(x\wedge 1)\frac{e^z}{V(|z|)}.
\end{eqnarray*}







{\em Case 2:} $\alpha<2$ and
$z\leq -1$, $1\leq x$.

For $y\le 1/2$ we have, by (\ref{case1}), $ G_{(0,\infty)}(x,y)\approx \frac {V(x)V(y)}x$ and similarly  $ G^{(\alpha)}_{(0,\infty)}(x,y)\approx \frac {V^{(\alpha)}(x)V^{(\alpha)}(y)}x$. Observing that $\nu(y-z)\approx \nu^{(\alpha)}(y-z)$, for $y>0$ and applying  $V(x)\approx V^{(\alpha)}(x)$, which follows from Lemma \ref{potential}, we obtain

\begin{eqnarray*}
\int^{1/2}_0 G_{(0,\infty)}(x,y)\nu(y-z)dy&\approx& \frac{V(x)}{x}\nu(z)\int^{1/2}_0V(y)dy\\
&\approx&\frac{V(x)}{x}\nu(z)\approx \int^{1/2}_0\frac{V^{(\alpha)}(x)V^{(\alpha)}(y)}{x}\nu^{(\alpha)}(y-z)dy\\
&\approx&\int^{1/2}_0{G}^{(\alpha)}_{(0,\infty)}(x,y)\nu^{(\alpha)}(y-z)dy.
\end{eqnarray*}
Note, by Lemma \ref{potential_lower} and (\ref{GcomphatG}), that
$G_{(0,\infty)}(x,y)\ge  {G}^{(\alpha)}_{(0,\infty)}(x,y)$,  and
$G_{(0,\infty)}(x,y)\approx \hat{G}^{(\alpha)}_{(0,\infty)}(x,y)\approx {G}^{(\alpha)}_{(0,\infty)}(x,y)$, for $y\geq \frac{1}{2}, |x-y|\ge 1$.
We then infer, there is $c$ such that

\begin{eqnarray*}  \int^{\infty}_0{G}^{(\alpha)}_{(0,\infty)}(x,y)\nu^{(\alpha)}(y-z)dy&\le &
\int^{\infty}_0 G_{(0,\infty)}(x,y)\nu(y-z)dy\\
&\le & c\int^{\infty}_0{G}^{(\alpha)}_{(0,\infty)}(x,y)\nu^{(\alpha)}(y-z)dy\\&+&
\int_{x-1}^{x+1} G_{(0,\infty)}(x,y)\nu(y-z)dy.
\end{eqnarray*}
Moreover, by (\ref{GreenExit}),
\begin{eqnarray*}  \int_{x-1}^{x+1} G_{(0,\infty)}(x,y)\nu(y-z)dy&\le& \nu(x-1-z) \int_{0}^{2x} G_{(0,\infty)}(x,y)dy\le c \nu(x-z)V(x) V(2x)  \\
&\approx &  \frac {V^2(x)}{|x-z|^{1+\alpha}}\le  c \frac {V(x)}{|x-z||z|^{\alpha/2}}\approx P^{(\alpha)}_{(0,\infty)}(x,z).
\end{eqnarray*}
which finally implies
$$P_{(0,\infty)}(x,z)\approx P^{(\alpha)}_{(0,\infty)}(x,z)\approx \frac {V(x)}{V(|z|)}\frac1{|x-z|},$$

{\em Case 3:} $\alpha< 2$ and $ z<-1$, $x\le 1$.



By (\ref{GreenExit}),

$$\int^{2}_{0} G_{(0,\infty)}(x,y)\nu(y-z)dy\le   {\nu (|z|)} \int^{2}_{0} G_{(0,\infty)}(x,y)dy\le c \nu (z) V(x)\approx  \frac{V(x)} {|z|^{1+\alpha}}.$$
For $y\ge 2$, by (\ref{case1}), we have $ G_{(0,\infty)}(x,y)\approx \frac {V(x)V(y)}y\approx \frac {V(x)}{y^{1-\alpha/2}}$ hence

$$\int_2^{\infty} G_{(0,\infty)}(x,y)\nu(y-z)dy\approx V(x) \int_2^{\infty} \frac{dy}{y^{1-\alpha/2}|y-z|^{1+\alpha}}\approx  \frac{V(x)} {|z|^{1+\alpha/2}},$$
which yields
\begin{eqnarray*}
P_{(0,\infty)}(x,z)&\approx&  \frac{V(x)}{ {|z|^{1+\alpha/2}}}
\approx
\frac{V(x)}{V(|z|)}\frac{1}{x-z}.
\end{eqnarray*}

{\em Case 4:} $\alpha\le 2$ and  $ -1<z<0$, $x\geq 1$.

We split the integral defining the Poisson kernel into three parts
\begin{eqnarray*}
P_{(0,\infty)}(x,z)&= &\int^{|z|/4}_0G_{(0,\infty)}(x,y)\nu(y-z)dy+\int_{|z|/4}^{1/2}G_{(0,\infty)}(x,y)\nu(y-z)dy\\&+& \int^{\infty}_{1/2}G_{(0,\infty)}(x,y)\nu(y-z)dy.
\end{eqnarray*}
For $y\le 1/2$, by (\ref{case1}), we have $ G_{(0,\infty)}(x,y)\approx \frac {V(x)V(y)}x$. Moreover $\nu(z)\approx \frac1{|z|}$,   hence

$$\int^{|z|/4}_0G_{(0,\infty)}(x,y)\nu(y-z)dy\approx \frac {V(x)}x \nu(z) \int^{|z|/4}_0 V(y)dy\le c \frac {V(x)}x. $$
Next, applying (\ref{IntV4}), the second integral is estimated in the following way

$$\int_{|z|/4}^{1/2}G_{(0,\infty)}(x,y)\nu(y-z)dy\approx \frac {V(x)}x  \int_{|z|/4}^{1/2} V(y)\nu(y)dy\approx  \frac {V(x)}x  \int_{|z|/4}^{1/2} \frac{V(y)}ydy \approx \frac {V(x)}{xV(|z|/4)}. $$
Summing both estimates we infer that

$$\int_{0}^{1/2}G_{(0,\infty)}(x,y)\nu(y-z)dy\approx  \frac {V(x)}{xV(|z|/4)} \approx \frac {V(x)}{V(|z|)}  \frac 1{|x-z|}. $$
For $y\le x/2 $ or $y\ge 2x $, by (\ref{case1}),  we have $ G_{(0,\infty)}(x,y)\approx \frac {V(x)V(y)}{x+y}$, hence applying (\ref{GreenExit}), we arrive at

\begin{eqnarray*}\int^{\infty}_{1/2}G_{(0,\infty)}(x,y)\nu(y-z)dy&\le&
c \frac{V(x)}x \int^{\infty}_{1/2}V(y)\nu(y)dy + \int^{2x}_{x/2}G_{(0,\infty)}(x,y)\nu(z-y)dy\\
&\le & c \frac{V(x)}x + \nu(x/2) \int^{2x}_{0}G_{(0,\infty)}(x,y)dy\\
&\le & c \frac{V(x)}x + \nu(x/2) V(x)V(2x)\\
&\approx &  \frac{V(x)}x.
\end{eqnarray*}
Combining   all the estimates of the integrals we obtain
\begin{eqnarray*}
P_{(0,\infty)}(x,z)&\approx & \frac{V(x)}{V(|z|)}\frac{1}{x-z},\  \text{for $\alpha\le 2$}.
\end{eqnarray*}
Noting that for $\alpha=2$ we have  $\frac{V(x)}{x-z}\approx 1$, we can rewrite the above comparison as
\begin{eqnarray*}
P_{(0,\infty)}(x,z)&\approx & \frac{1}{V(|z|)},\  \text{for $\alpha=2$}.
\end{eqnarray*}

\end{proof}


\begin{thm}\label{PoissonKernel}
Let $z<0<x$ and $\alpha\in(0,2]$, then
$$P_{(0,\infty)}(x,z)\approx\left\{%
\begin{array}{ll}\frac{V(x\wedge1)} {V(|z|)}\frac{1 }{(x-z) \log(1+\frac1{x-z})}e^z, & \hbox{$\alpha=2$,} \\
    \frac{V(x)} {V(|z|)}\frac{1 }{(x-z) \log(2+\frac1{x-z})}, & \hbox{$\alpha<2$,} \\
    \end{array}%
\right.$$

\end{thm}
\begin{proof}

By Lemma \ref{PoissonHat} it remains to consider the  case $-1<z<0<x<1$.  By Remark \ref{Greenhalflineasymp} we have
\begin{eqnarray*}
R(x,z)=\int^2_0G_{(0,\infty)}(x,y)\nu(y-z)dy&\approx&\int^2_0\left(1\wedge\frac{V(x)V(y)}{V^2(|x-y|)}\right)\frac{1}{|x-y|\log^2\left(1+\frac{1}{|x-y|}\right)}\frac{dy}{y-z}.
\end{eqnarray*}
 Let us denote
\begin{eqnarray*}
I_1&=&\int^{x/2}_0V(y)\frac{dy}{y-z},\\
I_2&=&\int^{3x/2}_{x/2}\frac{1}{|x-y|\log^2\left(1+\frac{1}{|x-y|}\right)}dy,\\
I_3&=&\int^2_{3x/2}\frac{V^3(y)}{y}\frac{dy}{y-z}.
\end{eqnarray*}
Note that
\begin{eqnarray*}R(x,z)&\approx& \frac{V^3(x)}{x}I_1+\frac{1}{x-z}I_2+V(x)I_3.\end{eqnarray*}
We start with the estimate of $I_2$,
\begin{eqnarray*}
I_2&\approx&\int^{3x/2}_{x}\frac{1}{(y-x)\log^2(y-x)}dy\approx \frac{1}{\log \left(1+\frac{1}{x}\right)}\approx V^2(x).
\end{eqnarray*}
For $|z|<x$, by (\ref{IntV4}),
$$I_3\approx \int^2_{3x/2}\frac{V^3(y)}{y^2}dy\approx \frac{V^3(x)}{x}.$$
Moreover, by (\ref{IntV1}),
\begin{eqnarray}I_1&\approx& \int^{x/2}_{|z|/4}\frac{1}{y\log^{1/2}(1+y^{-1})}dy+\frac{1}{|z|}\int^{|z|/4}_0 V(y)dy\approx \log^{1/2}(1+4/|z|)-\log^{1/2}(1+x^{-1})+V(|z|)\nonumber\\
&\approx&V(|z|)\log\left(1+\frac{x}{|z|}\right).\label{PoissonI}
\end{eqnarray}
Hence, for $x>|z|$,
\begin{equation*}\label{PoissonR}
R(x,z)\approx  \frac{V^2(x)}{x}\left(1+V(|z|)V(x)\log\left(1+\frac{x}{|z|}\right)\right).
\end{equation*}
Assume that $2|z|<x<1/2$, then
\begin{eqnarray}1+V(|z|)V(x)\log\left(1+\frac{x}{|z|}\right)&\approx&  1+ \frac 1{\log^{1/2}( \frac1{|z|})}\frac 1{\log^{1/2} (\frac1{x})}\log \frac x{|z|}\nonumber\\
&=& 1 + \frac {\log^{1/2}( \frac1{ |z|})}{\log^{1/2}( \frac1{ x})}- \frac {\log^{1/2}( \frac1{ x})}{\log^{1/2}( \frac1{ |z|})}
\approx \frac {\log^{1/2}( \frac1{ |z|})}{\log^{1/2}( \frac1{ x})}\nonumber\\
&\approx & \frac {V(x)}{V(|z|)}.\label{PoissonR1}
\end{eqnarray}
If  $|z|<x\leq 2|z|$, then \begin{equation}\label{PoissonR2}1+V(|z|)V(x)\log\left(1+\frac{x}{|z|}\right)\approx 1\approx \frac {V(x)}{V(|z|)}.\end{equation} For $x\geq 1/2$, we have
\begin{equation}\label{PoissonR3}1+V(|z|)V(x)\log\left(1+\frac{x}{|z|}\right)\approx V(z)\log\left(1+\frac{1}{|z|}\right)\approx \frac {V(x)}{V(|z|)}.\end{equation}
That is $$R(x,z)\approx \frac{V(x)}{V(|z|)}\frac{V^2(x)}{x}\approx \frac{V(x)}{V(|z|)}\frac{V^2(x-z)}{x-z}.$$
If $|z|\geq x$ we have by (\ref{IntV1}), $$I_1\approx \frac{1}{|z|}\int^{x/2}_0{V(y)}dy\approx \frac{x}{|z|}V(x),$$and  by (\ref{IntV4}),
\begin{eqnarray*}
I_3&\approx&\int^2_{7/4|z|}\frac{V^3(y)}{y^2}dy+\frac{1}{|z|}\int^{7/4|z|}_{3/2x}\frac{1}{y\log^{3/2}(1+y^{-1})}dy\\&\approx&\frac{V^3(|z|)}{|z|}+\frac{1}{|z|}\left(\frac{1}{\log^{1/2}(1+\frac{4}{7|z|})}-\frac{1}{\log^{1/2}(1+\frac{2}{3x})}\right)\\&\approx&\frac{V(|z|)}{|z|}\left(V^2(|z|)+V^2(x)\log\left(1+\frac{|z|}{x}\right)\right)\\&\approx& \frac{V^2(x)V(|z|)}{|z|} \log\left(1+\frac{|z|}{x}\right).
\end{eqnarray*}
Combining  the estimates of the integrals $I_1, I_2$ and $I_3$  we arrive at, for $x\geq|z|$,  $$R(x,z)\approx  \frac{V^2(x)}{|z|}\left(1+V(|z|)V(x)\log\left(1+\frac{|z|}{x}\right)\right).$$
By symmetry and (\ref{PoissonR1}-\ref{PoissonR3}) we infer that
$$R(x,z)\approx  \frac{V^2(x)}{|z|}\frac{V(|z|)}{V(x)}\approx \frac{V(x)}{V(|z|)}\frac{V^2(x-z)}{x-z}.$$
Next, observe that $R(x,-1)
\le R(x,z)$ hence from the above established bound and Lemma \ref{potential} we infer that
\begin{equation}\label{lowerR}R(x,z)\ge  c V(x).\end{equation}
Since, by (\ref{case1}), $ G_{(0,\infty)}(x,y)\approx \frac{V(x)V(y)}y, \ y>2$,  we obtain      
$$\int^\infty_2G_{(0,\infty)}(x,y)\nu(y-z)dy\leq c V(x) \int^\infty_2 \frac{V(y)}y\nu(y)dy\leq c V(x),$$
which together with (\ref{lowerR})
implies that  the Poisson kernel is comparable with $R(x,z)$.
Hence, by Lemma \ref{potential}
$$P_{(0,\infty)}(x,z)\approx \frac{V(x)}{V(|z|)}\frac{1}{(x-z)\log\left(2+\frac{1}{x-z}\right)}.$$
%

\end{proof}

\begin{rem}\label{Poissonasymp}
Let $-1<z<0<x<1$, then
$$P_{(0,\infty)}(x,z)\approx    \int^{7/4 (x\vee |z|)}_0 \left(1\wedge\frac{V(x)V(y)}{V^2(|x-y|)}\right)\frac{1}{|x-y|\log^2\left(1+\frac{1}{|x-y|}\right)}\frac{dy}{y-z}
    .$$
\end{rem}

\section{Boundary Harnack principle}
In this section we derive the Harnack inequality for non-negative  harmonic functions in intervals. The method we apply for this purpose is a regularization of the Poisson kernel of an interval or rather its upper bound provided by the Poisson kernel of a half-line. We follow the approach  of \cite{BSS}, where it was used to deal with a class of  symmetric stable processes not necessarily rotation invariant. As a consequence of the Harnack inequality we obtain the boundary Harnack  principle.

We start with two elementary lemmas, which we leave without rigorous proofs, giving only some explanation how to derive them. The first lemma follows from the Ikeda-Watanabe formula and the fact that  $\nu$ is radially decreasing.
\begin{lem}\label{PoissonExit}
For any $r>0$ and $|x|<r < |z|$,
$$E^x\tau_{(-r,r)}\nu(|z|+2r)\leq P_{(-r,r)}(x,z)\leq E^x\tau_{(-r,r)}\nu(|z|-r).$$
\end{lem}

From Proposition \ref{exptime}  we have   $E^x\tau_{(-r,r)} \approx V(r)V(r-|x|)$. Combining this with the above lemma and the properties of the L{\'e}vy measure we easily obtain the following estimates.
\begin{lem}\label{PoissonExit1}

Suppose that $h$ is a non-negative function. Let $ p>1$ and $r>0$. Let
$$ h_2(x)=E^x\left[h(X_{\tau_{(-r,r)}}), |X_{\tau_r}|>pr\right].$$
Then there is $C=C(p,\alpha)>0$ such that for $|x|< r$,

$$ C^{-1} V(r)V(r-|x|) \int_{|z|>pr}h(z)\nu(z)dz\le h_2(x)\le C V(r)V(r-|x|) \int_{|z|>pr}h(z)\nu(z)dz,\quad  0<\alpha<2,$$
and
$$ C^{-1}e^{-2r} V(r)V(r-|x|) \int_{|z|>pr}h(z)\nu(z)dz\le h_2(x)\le Ce^r V(r)V(r-|x|) \int_{|z|>pr}h(z)\nu(z)dz,\quad \alpha=2.$$


\end{lem}

\begin{thm}[Harnack inequality]\label{Harnack}
Let $1<p\le 3/2$. There exists a constant $C=C(\alpha, p)$  such that for any $r>0$ and any
nonnegative function $h$, harmonic in $(-2r, 2r)$, it holds, for  $0<\alpha<2$,
$$C^{-1} V^2(r) \int_{|z|>pr}h(z)\nu(z)dz\le h(x)\leq C V^2(r) \int_{|z|>pr}h(z)\nu(z)dz ,\qquad x\in (-r,r).$$
For $\alpha=2$ we have

$$C^{-1}e^{-5/2r} V^2(r) \int_{|z|>pr}h(z)\nu(z)dz\le h(x)\leq Ce^{2r} V^2(r) \int_{|z|>pr}h(z)\nu(z)dz ,\qquad x\in (-r,r).$$

\end{thm}

\begin{proof}
In the proof below the appearing constants $c_1, c_2,\dots$ will depend on $p, \alpha$, only.
 For simplicity, we will write $\tau_{(-r,r)}$ as $\tau_r$.
We start with  the upper bound. Define
$$\tilde{P}(x,z)=\int^{13/8r\wedge |z|}_{pr}P_{(-t,t)}(x,z)dt,\quad  |z|>pr.$$
Since $h$ is harmonic on $(-2r,2r)$, for all $t\in[pr, 13/8r]$, we have
$$h(x)=\int_{|z|>t}P_{(-t,t)}(x,z)h(z)dz.$$
Therefore
\begin{eqnarray*}
(13/8-p)rh(x)&=&\int^{13/8r}_{pr}\int_{|z|>t}P_{(-t,t)}(x,z)h(z)dzdt\\&=&\int_{pr<|z|<7/4r}\tilde{P}(x,z)h(z)dz+\int_{|z|>7/4r}\tilde{P}(x,z)h(z)dz\\&= & I_1+I_2.
\end{eqnarray*}

By Lemma \ref{PoissonExit1} we have
$$I_2\le c_1\left\{%
\begin{array}{ll} (13/8-p)rV^2(r)\int_{|z|>7/4r}h(z)\nu(z)dz , & \hbox{$\alpha<2$,} \\
    e^{2r}(13/8-p)rV^2(r)\int_{|z|>7/4r}h(z)\nu(z)dz, & \hbox{$\alpha=2$.} \\
    \end{array}%
\right.$$
In order to estimate $I_1$ we need an upper bound  of $\tilde{P}(x,z)$.
We claim that there is a constant $c_2$ such that for $pr<|z|<(7/4)r$, $|x|<r$,
\begin{equation}\label{Ptilde}\tilde{P}(x,z)\le c_2  V^2(r\wedge1).\end{equation}
By symmetry, we can assume that $z<-pr$. Then, we have
\begin{eqnarray*}
\tilde{P}(x,z)&\leq&\int^{13/8r\wedge |z|}_{pr}P_{(-t,\infty)}(x,z)dt=\int^{13/8r\wedge |z|}_{pr}P_{(0,\infty)}(x+t,z+t)dt.
\end{eqnarray*}
Since $|x|<r$, then $(p-1)r<x+t<3r$ and $x-z>(p-1)r$. First, assuming  $r\leq 1$ for $\alpha=2$ or arbitrary $r$ for $0<\alpha<2$,  by Theorem \ref{PoissonKernel},

 \begin{eqnarray*}
\tilde{P}(x,z)&\leq& c_3 \frac{1 }{(x-z) \log(2+\frac1{x-z})}\int^{13/8r\wedge |z|}_{pr}
\frac{V(x+t)}{V(|z|-t|)}dt\\&\leq& c_4 \frac{1 }{(x-z) \log(2+\frac1{x-z})} V(3r)\int^{ |z|}_{0}\frac {dt}{V(t)}\\
&\leq& c_5 \frac{1 }{r \log(2+\frac1r)} V(r)\int^{ 2r}_{0}\frac {dt}{V(t)}\\
\end{eqnarray*}
Noting that  $ \int^{ 2r}_{0}\frac {dt}{V(t)}\approx \frac r{V(r)}$ we obtain
$$\tilde{P}(x,z)\le c_6 \frac{1 }{ \log(2+\frac1r)}\approx V^2(r\wedge1).$$
Similarly, for $\alpha=2$ and $r\ge 1$,
\begin{eqnarray*}
\tilde{P}(x,z)&\leq& c_7 \int^{13/8r\wedge |z|}_{pr}\frac{V(1)}{V(|z|-t)}e^{-|z|+t}dt
\leq c_7V(1)\int^\infty_0\frac{e^{-u}}{V(u)}du.
\end{eqnarray*}

By (\ref{Ptilde}) and since the   density of the Levy measure is radially decreasing we have

$$r^{-1}I_1= r^{-1}\int_{pr<|z|<7/4r}\tilde{P}(x,z)h(z)dz\le c_2\frac{ V^2(r\wedge1)}{r\nu(r)}\int_{pr<|z|<7/4r}\nu(z)h(z).$$
Note that, for $\alpha<2$ we have $\frac{ V^2(r\wedge1)}{r\nu(r)}\approx V^2(r)$, and for $\alpha=2$ we have $\frac{ V^2(r\wedge1)}{r\nu(r)}\approx V^2(r)\frac {e^r}{1+r}$.
Combining this with the above estimates of $I_1$ and $I_2$ we obtain
$$h(x)\le c_{10} V^2(r) \int_{|z|>pr}h(z)\nu(z)dz,\quad  0<\alpha<2,$$
$$h(x)\le c_{10} e^{2r} V^2(r) \int_{|z|>pr}h(z)\nu(z)dz,\quad  \alpha=2.$$

Finally we find the lower bound for $h(x)$.
Let $q=(1+p)/2$. Next, for
$$ h_2(x)=E^x\left[h(X_{\tau_{qr}}), |X_{\tau_{qr}}|>pr\right],$$
by Lemma \ref{PoissonExit1}, for $0<\alpha<2$, we arrive at
$$h(x)\ge  h_2(x)\ge c_{11} V^2(r) \int_{|z|>pr}h(z)\nu(z)dz.$$ Similarly, for  $\alpha=2$, we have
$$h(x)\ge  h_2(x)\ge c_{11} e^{-5/2r} V^2(r) \int_{|z|>pr}h(z)\nu(z)dz.$$

\end{proof}

\begin{rem}
The weak form of the Harnack inequality for the geometric stable process  was proved in  $\Rd$ for $d>\alpha$ in \cite{SV}, Theorem 6.6. It was  shown there that there is a constant $C=C(r,\alpha, d)$ such that
 for any harmonic function $h$ in a ball $B(0,r)$ we have
 $$h(x)\le C  h(y), \quad x,y\in B(0,r/2).$$ As a function of $r>0$ the obtained constant C tends to $\infty$ as $r\searrow 0$. Such form is not scale invariant. The constant from Theorem \ref{Harnack} does not depend on $r$ for $0<\alpha<2$. In the last section we find two-sided estimates for  the Poisson kernel of any interval which allow to improve the Harnack inequality for $\alpha=2$ (see Theorem \ref{Harnack_imp}).
 Our result suggests that the scale invariant version can be proved also in  higher dimensions.

\end{rem}

\begin{thm}
[Boundary Harnack property]\label{BHP}
There exists a constant $C=C(\alpha)$  such that for any $r>0$ and any
nonnegative function $h$ regular harmonic in $(0, 2r)$ which vanishes in $(-2r,0)$ we have, for $\alpha=2$,
$$C^{-1} e^{-2r} \frac {V(x)} {V(r)}\le  \frac {h(x)} {h(r)}\le C e^{2r} \frac {V(x)} {V(r)} ,\quad  0<x<r.$$
For $0<\alpha<2$,
$$C^{-1}  \frac {V(x)} {V(r)}\le  \frac {h(x)} {h(r)}\le C  \frac {V(x)} {V(r)} ,\quad  0<x<r.$$

\end{thm}

\begin{proof} We provide the proof for the case $0<\alpha<2$, only.
Let
$ h_2(x)=E^x\left[h(X_{\tau_{r}}), |X_{\tau_{r}}|>3/2r\right]$.
Note that by the Harnack inequality and Lemma \ref{PoissonExit1} we have  $ h_2(r/2)\approx h(r/2)\approx h(r)$. Moreover, by Lemma \ref{PoissonExit1}, we have
$$\frac {h_2(x)} {h_2(r)}\approx \frac {V(x)} {V(r)}.$$
Hence,
$$h_2(x)\approx h(r) \frac {V(x)} {V(r)}.$$
Next, by the Harnack inequality
$$ h_1(x)=E^x\left[h(X_{\tau_{r}}), r\le X_{\tau_{r}}<3/2r\right]\le C h(r)P^x(r<X_{\tau_{r}}<3/2r)\le C h(r) \frac {V(x)} {V(r)}. $$
This implies that
$$h(x)=h_1(x)+h_2(x)\approx h(r) \frac {V(x)} {V(r)} ,\quad  0<x<r.$$

\end{proof}

\section{Green function and Poisson kernel of the interval}

This section is devoted to extension of  the results of Section 4 to intervals. We show optimal estimates of the Green functions and Poisson kernels for intervals taking into account the size of intervals. Note that by passing to infinity with the length of intervals we recover the estimates from Section 4. This does not mean that the results of Section 4 can be obtained from the current section.  In fact, we strongly use the estimates for half-lines, showing that for some choice of variables and interval lengths the Green functions and Poisson kernels are comparable for intervals and intervals.

\begin{lem}\label{GRGinfty}

A) There exists a constant $a\leq 1/2$ such that, for $0<x,y\leq aR$,
$$G_{(0,R)}(x,y)\geq 1/2 G_{(0,\infty)}(x,y).$$

B) For any $0<a <1/2$ there is a constant $b<a/2$ such that, for $R\leq 4$, and $a/2R<x<y<(1-a/2)R$,
$$G_{(0,R)}(x,y)\geq 1/2 G_{(0,\infty)}(x,y),$$
if $|x-y|\leq bR$.
\end{lem}
\begin{proof} Throughout the whole proof we assume that $0<x<y$ and $a<1/2$.
Denote $\tau_R=\tau_{(0,R)}$ and observe that
\begin{eqnarray*}
G_{(0,R)}(x,y)&=&G_{(0,\infty)}(x,y)-E^xG_{(0,\infty)}(X_{\tau_{R}},y).
\end{eqnarray*}
Note that $G_{(0,\infty)}(z,y)$ is decreasing on $(y,\infty)$ as a function of $z$, which together with Lemma \ref{exit_upper} implies
\begin{equation}\label{fraction0}
E^xG_{(0,\infty)}(X_{\tau_{R}},y)\leq  G_{(0,\infty)}(R,y)\frac {V(x)}{V(R)}.\end{equation}

Observe that for $x<y$ we have
\begin{equation}\label{V}
\frac {V(x)V(y)} {1\wedge\frac{V(x)V(y)}{V^2(|y-x|)}}\le V^2(y).
\end{equation}



Suppose that  $x,y\le 2\wedge aR$. Then by Remark \ref{Greenhalflineasymp},

$$G_{(0,\infty)}(x,y)\approx \left(1\wedge\frac{V(x)V(y)}{V^2(|y-x|)}\right)|y-x|^{-1}\log^{-2}(1+|y-x|^{-1}).$$
 By (\ref{case1}), we have  $ G_{(0,\infty)}(R,y)\approx \frac {V(y)V(R)}{R\log(2+R^{-1})}$. Applying (\ref{V}) we obtain the following bound


\begin{eqnarray*}\frac{\frac{V(x)}{V(R)}{G}_{(0,\infty)}(R,y)}{G_{(0,\infty)}(x,y)}&\le& c  \frac{V^2(y)|y-x|\log^{2}(2+\frac{1}{|y-x|})}{R\log(2+R^{-1})}\\&\le&
 c  \frac{V^2(aR\wedge1) (aR\wedge1 )\log^{2}(2+\frac{1}{aR\wedge1})}{R\log(2+R^{-1})}.\end{eqnarray*}
%
%
%
%
Next,     by Lemma \ref{potential}, we infer that

$${V^2(aR\wedge1) \approx \log^{-1}(2+\frac{1}{aR\wedge1})},$$
which proves that

\begin{eqnarray}\frac{\frac{V(x)}{V(R)}{G}_{(0,\infty)}(R,y)}{G_{(0,\infty)}(x,y)}&\le&
 c  \frac{ (aR\wedge1 )\log(2+\frac{1}{aR\wedge1})}{R\log(2+R^{-1})}\nonumber\\
 &\le& c a \log(2+\frac{1}{a}).\label{fraction1}\end{eqnarray}


 Assume now that $x<1<2<y<aR$ or $1<x<y<aR$ . If $x<1<2<y<aR$, due to (\ref{case1}),  $G_{(0,\infty)}(x,y)\approx \frac{ V(x) V(y)}y$.  If $1<x<y<aR$ then

 $$G_{(0,\infty)}(x,y)\ge c\hat{G}^{(\alpha)}_{(0,\infty)}(x,y)\geq c \frac{V(x)V(y)}{y}.$$

By (\ref{case1}) we have  $G_{(0,\infty)}(y,R)\approx \frac{ V(y) V(R)}R$ which together with $G_{(0,\infty)}(x,y)\ge c\frac{V(x)V(y)}{y}$
 implies

\begin{equation}\label{fraction2}\frac{G_{(0,\infty)}(R,y)\frac {V(x)}{V(R)}}  { G_{(0,\infty)}(x,y)}\le c \frac {y}{R}\le c a.  \end{equation}

Combining (\ref{fraction0}, \ref{fraction1}, \ref{fraction2}) we infer that

$$E^xG_{(0,\infty)}(X_{\tau_{R}},y)\leq c a \log(2+\frac{1}{a}) G_{(0,\infty)}(x,y)\le (1/2)G_{(0,\infty)}(x,y)$$
for sufficiently small $a$, which completes the proof of the first part of the Lemma.

Now we proceed with the proof of  part $(B)$. Let $R\leq 4$ and $a/2R<x<y<(1-a/2)R$. Assume that $|x-y|\leq bR $. Let us observe that $V^2(y-x)\leq V^2(bR)\leq V^2(a/2R)\leq V(x)V(y)$. Then by Remark \ref{Greenhalflineasymp} we have
\begin{eqnarray*}
\frac{\frac{V(x)}{V(R)}G_{(0,\infty)}(R,y)}{G_{(0,\infty)}(x,y)}&\approx&\frac{\frac{V(x)}{V(R)}\left(1\wedge\frac{V(R)V(y)}{V^2(R-y)}\right)\frac{1}{(R-y)\log^2(1+(R-y)^{-1})}}{\left(1\wedge\frac{V(x)V(y)}{V^2(y-x)}\right)\frac{1}{(y-x)\log^2(1+(y-x)^{-1})}}\\
&\leq& \frac{\frac{1}{(R-y)\log^2(1+(R-y)^{-1})}}{\frac{1}{(y-x)\log^2(1+(y-x)^{-1})}}\leq c\frac{bR\log^2\left(1+\frac{1}{bR}\right)}{aR\log^2\left(1+\frac{1}{aR}\right)}\\&\leq&c \frac{b}{a}\left(1+\log\frac{a}{b}\right)^2.
\end{eqnarray*}
 Hence $$E^xG_{(0,\infty)}(X_{\tau_{R}},y)\le (1/2)G_{(0,\infty)}(x,y)$$
for sufficiently small $b$.

\end{proof}

Standard arguments imply the estimates of the Green function of the interval $(0, R), \ R>0$ if $R$ is bounded by a fixed positive number $R_0$. In the theorem below we choose $R_0=4$ as an upper bound for $R$, however we could chose any positive number at the expense of the comparability constant.
\begin{thm}\label{GreenSmallInterval}
Let $R< 4$ then we have,
$$G_{(0,R)}(x,y)\approx \left(1\wedge\frac{V(\delta_R(x))V(\delta_R(y))}{V^2(|y-x|)}\right)|y-x|^{-1}\log^{-2}(1+|y-x|^{-1})
    .$$
		\end{thm}
\begin{proof}
If $x,y<aR$ then, by Theorem \ref{Greenhalf} and  Lemma \ref{GRGinfty} we
get $$G_{(0,R)}(x,y)\approx G_{(0,\infty)}(x,y)\approx  \left(1\wedge\frac{V(x)V(y)}{V^2(|y-x|)}\right)|y-x|^{-1}\log^{-2}(1+|y-x|^{-1}).$$
By symmetry we have, for $x,y>(1-a)R$,
$$G_{(0,R)}(x,y)\approx\left(1\wedge\frac{V(R-x)V(R-y)}{V^2(|y-x|)}\right)|y-x|^{-1}\log^{-2}(1+|y-x|^{-1}).$$
Let $a/2R<x<y<(1-a/2)R$. If $|x-y|\leq bR$  then again, by Lemma \ref{GRGinfty} and Theorem \ref{Greenhalf},
$$G_{(0,R)}(x,y)\approx |y-x|^{-1}\log^{-2}(1+|y-x|^{-1}).$$
If $R>|x-y|> bR$,  the Harnack inequality implies
$$G_{(0,R)}(x,y)\approx G_{(0,R)}(y-bR, y)\approx  |y-x|^{-1}\log^{-2}(1+|y-x|^{-1}).$$
For $x<a/2R$ and $y>aR$ we use the boundary Harnack principle to get
$$G_{(0,R)}(x,y)\approx \frac{V(x)}{V(a/2R)}G_{(0,R)}(a/2R,y).$$
If $y>(1-a/2)R$ we again use the boundary Harnack principle
$$G_{(0,R)}(x,y)\approx \frac{V(x)}{V(a/2R)}G_{(0,R)}(a/2R,(1-a/2)R)\frac{V(y)}{V((1-a/2)R)}.$$
Hence $$G_{(0,R)}(x,y)\approx \frac{V(\delta_R(x))V(\delta_R(y))}{V^2(|y-x|)}|y-x|^{-1}\log^{-2}(1+|y-x|^{-1}).$$
\end{proof}

To extend the above  uniform bound to large intervals we  define a function $\hat{G}_{(0,R)}^{(\alpha)}(x,y),\ x,y \in (0,R)$, such that
\begin{equation*}
\hat{G}_{(0,R)}^{(\alpha)}(x,y)=
\left\{%
\begin{array}{ll}
        \min\left\{\frac{1}{|x-y|^{1-\alpha}},
    \frac{V(\delta_R(x))V(\delta_R(y))}{|x-y|}\right\}, & \hbox{$\alpha<1$,} \\
    \ln\left(1+\frac{V(\delta_R(x))V(\delta_R(y))}{|x-y|}\right), & \hbox{$\alpha=1$,} \\
        \min\left\{\frac{V(\delta_R(x))V(\delta_R(y))}{(\delta_R(x)\delta_R(y))^{1/2}},
    \frac{V(\delta_R(x))V(\delta_R(y))}{|x-y|}\right\}, & \hbox{$1<\alpha<2$,} \\
		\frac{(V(x)V(R-y))\wedge (V(R-x)V(y))}{R},&\hbox{$\alpha=2$.}
\end{array}%
\right.
\end{equation*}
\begin{thm}\label{GreenLargeInterval}
Let $R\ge 4$ and $x\le y$ then we have for $|x-y|\le 1$,
$$
G_{(0,R)}(x,y)\approx \min\{G_{(0,\infty)}(x,y),
G_{(0,\infty)}(R-x,R-y) \} $$ and for $|x-y|> 1$
$$
G_{(0,R)}(x,y)\approx \hat{G}_{(0,R)}^{(\alpha)}(x,y). $$
\end{thm}
\begin{proof}
For $\alpha=2$ we can use similar methods to the proof of Theorem 6.4 in \cite{GR3}.
Therefore we assume that $\alpha<2$. By symmetry we have $G_{(0,R)}(x,y)=G_{(0,R)}(R-x,R-y)$ and we can assume that $x\leq y$. Hence $G_{(0,R)}(x,y)\leq \min\{G_{(0,\infty)}(x,y),
G_{(0,\infty)}(R-x,R-y) \}$. Let $|x-y|\leq 1$, and $x\leq R/2$. Then $\delta_R(y)\geq y/2$ and by Theorem \ref{GreenSmallInterval} we infer
\begin{eqnarray*}G_{(0,R)}(x,y)&\geq& G_{((x-2)\vee 0,(x+2)\vee R)}(x,y)\\&\approx& \left(1\wedge\frac{V(x\wedge2)V(y\wedge2)}{V^2(|x-y|)}\right)\frac{1}{|x-y|\log^2(1+\frac{1}{|x-y|})}.\end{eqnarray*}
Hence, Remark \ref{Greenhalflineasymp}, for $x\leq 1$, and Theorem \ref{Greenhalf}, Lemma \ref{potential_lower} and (\ref{GcomphatG}), for $x>1$, imply
$$G_{(0,R)}(x,y)\geq c G_{(0,\infty)}(x,y).$$
For $x>R/2$, we use symmetry to get $$G_{(0,R)}(x,y)\geq c G_{(0,\infty)}(R-x,R-y),$$
which proves, for $|x-y|\leq 1$,
$$ G_{(0,R)}(x,y)\approx \min\{G_{(0,\infty)}(x,y),
G_{(0,\infty)}(R-x,R-y).$$
Assume that $|x-y|>1$. Let us observe that, for $x,y\leq 3/4R$, we have $\hat{G}^{(\alpha)}_{(0,R)}(x,y)\approx \hat{G}^{(\alpha)}_{(0,\infty)}(x,y)$. Hence, by Remark \ref{Greenhalflineasymp}
$$G_{(0,R)}(x,y)\leq c \hat{G}_{(0,R)}^{(\alpha)}(x,y).$$
Lemma \ref{potential_lower} implies, for $x,y\geq1/2$,
\begin{equation}\label{GlarR1}G_{(0,R)}(x,y)\approx \hat{G}_{(0,R)}^{(\alpha)}(x,y).\end{equation}
If $x<1/2$ we use the boundary Harnack principle to get the above estimate. By symmetry, (\ref{GlarR1}) is true, for $x,y\geq R/4$, as well.
For $x<R/4$ and $y>3/4 R$ the boundary Harnack principle implies
$$G_{(0,R)}(x,y)\approx \frac{V(x)}{V(R/4)}G_{(0,R)}(R/4,3/4R)\frac{V(R-y)}{V(3/4R)}\approx  \hat{G}_{(0,R)}^{(\alpha)}(x,y).$$
\end{proof}

 Now, we prove estimates for the Poisson kernel of the interval $(0,R)$. By symmetry, $P_{(0,R)}(x,z)=P_{(0,R)}(R-x,R-z)$. Therefore it is enough to prove estimates for $z<0$ and $x\in(0,R)$.

\begin{thm}\label{Poissoninterval}Assume that $z<0<x<R$.

 For $0<\alpha \le 2$
and $x,|z|\leq 2\wedge R$  we have

$$P_{(0,R)}(x,z)\approx \frac{V(x)} {V(|z|)}\frac{V(R-x) }{V(R+|z|)}\frac1{(x-z) \log(2+\frac1{x-z})}. $$

For $0<\alpha<2$, when  $x> 2$ or $|z|> 2\wedge R$, we have
$$P_{(0,R)}(x,z)\approx\frac{V(x)V(R-x)}{V(|z|)V(|z|+R)}\frac{1}{x-z}.$$



For $\alpha=2$,
$$P_{(0,R)}(x,z)\approx \left\{
\begin{array}{ll}e^{-|z|}\frac{V(x\wedge1)V(R-x)}{RV(|z|)}, & \hbox{$R\ge 4$, $x>2$ or $|z|> 2$,} \\
   e^{-|z|}\frac{V(x)V(R-x) }{|z|}, & \hbox{$R\le 4$, $|z|\ge R $}. \\
    \end{array}%
\right.$$



\end{thm}

\begin{proof} We present arguments only for $\alpha=2$,  since  the proof is similar to the proof of Theorem \ref{PoissonKernel}. Moreover, for  intervals of length not bigger than $R_0=4$ (the upper bound $4$ can be replaced by any $R_0$ at the expense of the comparability constant), the proof below is suitable for all $\alpha$'s, provided that $|z|\le R$.

We will use Theorems \ref{GreenSmallInterval} and \ref{GreenLargeInterval}, therefore first we will prove estimates for $R\leq4$ and next for $R> 4$.

We start with $R\leq 4$. Assume additionally that $-R/2<z<0$. Clearly 
$$P_{(0,R)}(x,z)
\leq P_{(0,\infty)}(x,z).$$
Note that $V(R-y)\approx V(y)$, for $R/2\leq y\leq 7/8R$. Therefore by Theorem \ref{GreenSmallInterval} and Remark \ref{Poissonasymp}, for $x\leq R/2$, we have
\begin{equation}\label{PoissonInt2}
P_{(0,R)}(x,z)\geq c\int^{7/8R}_0\left(1\wedge\frac{V(x)V(y)}{V^2(|y-x|)}\right)\frac{1}{|y-x|\log^{2}(1+|y-x|^{-1})}\frac{dy}{y-z}\geq c P_{(0,\infty)}(x,z),
\end{equation}
yielding \begin{equation}\label{PoissonInt3}P_{(0,R)}(x,z)\approx P_{(0,\infty)}(x,z).\end{equation}
For $x>R/2$, we have $x-z\approx R\approx R/2-z$ and $V(x)\approx V(R/2)$, hence by the boundary Harnack principle,  (\ref{PoissonInt3}) and Theorem \ref{PoissonKernel},
\begin{eqnarray*}
P_{(0,R)}(x,z)&\approx&P_{(0,R)}(R/2,z)\frac{V(R-x)}{V(R/2)}\\
&\approx&\frac{V(R)}{(x-z)\log(2+\frac1{x-z})}\frac{V(R-x)}{V(R/2)}\\
&\approx&\frac{V(R-x)}{(x-z)\log(2+\frac1{x-z})}.
\end{eqnarray*}
The last comparability, (\ref{PoissonInt3}) and Theorem \ref{PoissonKernel}  imply that

$$P_{(0,R)}(x,z)\approx \frac{V(x)} {V(|z|)}\frac{V(R-x) }{V(R+|z|)}\frac1{(x-z) \log(2+\frac1{x-z})}, \quad -R/2<z<0<x<R\le 4. $$
%
%
%
For $z<-R/2$, we have $\nu(|z|)\approx \nu(z+3/2R)$. Hence, by Lemma \ref{PoissonExit} and Proposition \ref{exptime} we get, for $z\leq-R/2$,  \begin{equation*}\label{PoissonInt1}P_{(0,R)}(x,z)\approx E^x\tau_{(0,R)}\nu(z)\approx {V(x)V(R-x)}{\nu(z)}.\end{equation*}
This ends the proof for $R\le 4$.

Assume that $R\geq 4$. If $-1<z<0<x<1$ then by (\ref{PoissonInt3}),
$$CP_{(0,\infty)}(x,z)\leq P_{(0,3)}(x,z)\leq P_{(0,R)}(x,z)\leq P_{(0,\infty)}(x,z),$$
which, by Theorem \ref{PoissonKernel},  yields
$$P_{(0,R)}(x,z)\approx \frac{V(x)} {V(|z|)}\frac{V(R-x) }{V(R+|z|)}\frac1{(x-z) \log(2+\frac1{x-z})}, \quad -1<z<0<x<1. $$
For $x\vee|z|\geq 1$ and $x\leq R/2$ we use the same arguments like in the proof  Lemma \ref{PoissonHat} to get
$$P_{(0,R)}(x,z)\geq c
       e^z\frac{V(x\wedge1)}{V(|z|)}\approx P_{(0,\infty)}(x,z).$$
 %
%
Hence, $$P_{(0,R)}(x,z)\approx P_{(0,\infty)}(x,z)\approx   e^{-|z|}\frac{V(x\wedge1)V(R-x)}{RV(|z|)}   .$$


Now, assume that $x> R/2$. Denote $W(x,z)=\int^R_0\hat{G}^{(2)}(x,y)\nu(y-z)dy$, then by (\ref{IntV3})
\begin{eqnarray*}
W(x,z)&=&\int^x_0\frac{V(R-x)V(y)}{R}\nu(y-z)dy+\int^R_x\frac{V(x)V(R-y)}{R}\nu(y-z)dy\\
&=&\frac{e^z}{R}\left(V(R-x)\int^x_0\frac{V(y)e^{-y}}{y-z}dy+V(x)\int^{R-x}_0\frac{V(y)e^{-R+y}}{R-y-z}dy \right)\\
&\approx&\frac{e^z}{R}\left(V(R-x)\frac{1}{V(|z|)}+V(x)e^{-x}\frac{(1\wedge(R-x))V(R-x)}{R-z} \right)\\
&\approx&e^z\frac{V(R-x)}{RV(|z|)}.
\end{eqnarray*}

Moreover
$$\int^{R\wedge(x+1)}_{x-1}\left(1\wedge\frac{V(R-x)V(R-y)}{V^2(|x-y|)}\right)\frac{\nu(y-z)dy}{|x-y|\log^2(1+|x-y|^{-1})}\leq c P_{(0,\infty)}(R-x,R-z)\leq c W(x,z).$$
Hence, by Theorem \ref{GreenLargeInterval} $$P_{(0,R)}(x,z)\approx e^z\frac{V(R-x)}{RV(|z|)}\approx   e^{-|z|}\frac{V(x\wedge1)V(R-x)}{RV(|z|)}.$$\end{proof}
The next result is an improvement of the Harnack inequality for $\alpha=2$, that was proved in the previous section.
\begin{thm}\label{Harnack_imp}
There exists a constant $C=C(\alpha)$  such that for any $r>0$ and any
nonnegative function $h$, harmonic in $(-2r, 2r)$, it holds, for  $0<\alpha\leq2$,
$$h(x)\leq C h(y) ,\qquad x,y\in (-r,r).$$
\end{thm}
\begin{proof}
By Theorem \ref{Harnack} it is enough to prove the scale invariant Harnack inequality for $\alpha=2$ when $r> 4$.
We use Theorem \ref{Poissoninterval} to get   $$P_{(-3/2r,3/2r)}(x,z)\approx  e^{-(|z|-3/2r)}\frac{1}{V(|z|-3/2r)} \approx P_{(-3/2r,3/2r)}(0,z),\quad  |x|<r, |z|>3r/2,$$
which yields
\begin{eqnarray*}h(x)&=&\int_{|z|>3/2r}P_{(-3/2r,3/2r)}(x,z)h(z)dz\approx \int_{|z|>3/2r}P_{(0,R)}(0,z)h(z)dz\\
&=&h(0).\end{eqnarray*}
Hence,  $h(x)\approx h(y)$, for any $x,y \in (-r,r)$.
\end{proof}

\end{document}